\documentclass[11pt,twoside]{article}
\usepackage{a4}
\usepackage{amssymb,amsmath,amsthm,latexsym}
\usepackage{amsfonts}
\usepackage{amsmath}
\usepackage{amsfonts}
\usepackage{float}
\usepackage{geometry}
\usepackage{graphicx}
\usepackage{placeins}
\usepackage{longtable}
\usepackage[tableposition=below]{caption}
\usepackage{cite}
\usepackage{tikz}
\newtheorem{theorem}{Theorem}[section]

\newtheorem{lemma} [theorem]{Lemma}

\begin{document}
	\title{Extremal Values of the Atom-Bond Connectivity Index for Trees with Given Roman Domination Numbers
	}
	\author{Waqar Ali$^{1}$, Mohamad Nazri Bin Husin$^{1}$, Muhammad Faisal Nadeem$^{2}$}
	\date{}
	\maketitle
	\vspace{-7mm}
	\begin{center}
		
		{\it\small 1 Special Interest Group on Modeling and Data Analytics, Faculty of Computer Science and Mathematics, Universiti Malaysia Terengganu, Kuala Nerus 21030, Terengganu}\\
		{\it\small 2 Department of Mathematics, COMSATS University Islamabad Lahore Campus, Lahore 54000 Pakistan.}\\
		{meharwaqaraali@gmail.com\\}
		
		{corresponding author nazri.husin@umt.edu.my\\}
		
		{mfaisalnadeem@ymail.com\\}
		
	\end{center}
	\begin{abstract}
		Consider that $\mathbb{G}=(\mathbb{X}, \mathbb{Y})$ is a simple, connected graph with $\mathbb{X}$ as the vertex set and $\mathbb{Y}$ as the edge set. The atom-bond connectivity ($ABC$) index is a novel topological index that Estrada introduced in Estrada et al. (1998). It is defined as
		$$
		A B C(\mathbb{G})=\sum_{xy \in Y(\mathbb{G})} \sqrt{\frac{\zeta_x+\zeta_y-2}{\zeta_x \zeta_y}}
		$$
		where $\zeta_x$ and $\zeta_x$ represent the degrees of the vertices $x$ and $y$, respectively. In this work, we explore the behavior of the $A B C$ index for tree graphs. We establish both lower and upper bounds for the $A B C$ index, expressed in terms of the graph's order and its Roman domination number. Additionally, we characterize the tree structures that correspond to these extremal values, offering a deeper understanding of how the Roman domination number ($RDN$) influences the $A B C$ index in tree graphs.

	\end{abstract}
	\textbf{Key words:} $ABC$ index; Roman domination number, Tree, chemical graph theory, Extremal values.
	\section{Introduction}
	
	~~~~In theoretical chemistry, mathematical chemistry is the study of chemistry without reference to quantum mechanics, aimed at explaining and predicting the properties of molecules. Graph theory is used to represent chemical events in the important topic of chemical graph theory within mathematical chemistry. The chemical sciences have advanced significantly as a result of this strategy.
	
	A molecular graph is a fundamental graph where the edges represent bonds between atoms, and the vertices represent the atoms themselves. Hydrogen atoms are often omitted from these representations. According to IUPAC terminology, a topological index is a number associated with the chemical structure. This value is used to establish correlations between different physical and chemical attributes, biological activity, and chemical reactivity, and the chemical structure, see \cite{karelson1996quantum, gutman1972graph, kier2012molecular, kier1986molecular, todeschini2008handbook}.
	
	The Randi\'{c} connectivity index is one of the most well-known topological indices, supported by a solid mathematical foundation and widely applied in pharmacology and chemistry. In 1998, Estrada et al. \cite{estrada1998atom} introduced the $ABC(\mathbb{G})$ index as a notable alternative to the Randi\'{c} index. According to Furtula \cite{furtula2016atom}, the $ABC$ index ranks among the leading degree-based molecular descriptors.

	
	\begin{equation}\label{equ1}
		A B C(\mathbb{G})=\sum_{xy \in Y(\mathbb{G})} \sqrt{\frac{\zeta_x+\zeta_y-2}{\zeta_x \zeta_y}}
	\end{equation}
	
	\noindent where $Y(\mathbb{G})$ represents the set of edges in the graph $\mathbb{G}$, and $\zeta_x$ and $\zeta_y$ are the degrees of the vertices $x$ and $y$ connected by the edge $xy$. The index has been widely used to predict molecular stability, boiling points, and other properties, and has proven to be a valuable tool in QSAR/QSPR studies and drug design. Chen and Das \cite{chen2018solution} confirmed the conjecture that the Tur\'{a}n graph maximizes the $ABC$ index among graphs with a given chromatic number, resolving a problem posed by Zhang et al. \cite{zhang2016maximum}. 
	Zheng et al. \cite{zheng2020bounds} established bounds for the general $ABC$ index for connected graphs with fixed maximum degree, characterizing the extremal graphs for specific parameter ranges. Numerous studies have been conducted on this topological indicator, and research on it is still ongoing, see \cite{chen2012atom, lin2017minimal, vassilev2012minimum, lin2013minimal, chen2012some}.

	Prior research \cite{das2010comparison, xu2012relationships, das2016comparison} has investigated the relations between different topological indices and the $ABC$ index. Among other notable studies, Das and Trinajsti\'{c} \cite{das2010comparison} looked at the relationship between the $GA$ and $ABC$ index. Xinli Xu \cite{xu2012relationships} established correlations between the Harmonic index and several other indices, including the $ABC$ index, Randi\'{c} index, and the first Zagreb index, based on the order, size, and number of pendant vertices in the graph. Further research was done on the connection between the $ABC$ index and the distance-based variation of the $ABC_{\mathbb{G}}$ index by Das et al. \cite{das2016comparison}. Additionally, they determined which extremal trees reach these limits. In order to determine the structures of extremal graphs, Zhang et al. \cite{zhang2024extremal} investigated the extremal limits of the $ABS$ index for trees with certain matching and dominance values. In order to find graphs with the highest and least values of these indices, Wang et al. \cite{wang2018maximizing} looked into extremal multiplicative Zagreb indices in graphs with provided vertices and cut edges. Jamri et al. \cite{hasni2021randic, ahmad2022minimum} established both the extreme values of the $RI(\mathbb{T})$ with a specified $TDN$. Most recently, Bermudo et al. \cite{bermudo2024geometric, bermudo2020extremal} provided an maximum value for the $GA(\mathbb{T})$, based on their order and $TDN$, as well as both extreme values for the $RI(\mathbb{T})$, based on their order and domination number, and research on the bounds of various topological indices with given parameters is still ongoing, as noted in \cite{li2024greatest, wang2024maximum, noureen2024tricyclic, ahmad2023zagreb}.
	
	In this study, we take a look at a simple, uncorrected connected graph $\mathbb{G}$, which has a set of vertices $\mathbb{X}$ and an edge set $\mathbb{Y}$. An edge in graph $\mathbb{G}$ connecting two vertices, $x$ and $y$, is represented by the symbol $xy$. The open neighborhood of any vertex $y \in \mathbb{X}$ is defined as $N(y) = {x \in \mathbb{X} \mid xy \in \mathbb{Y}}$, whereas $N[y] = N(y) \cup {y}$ indicates the closed neighbor. The size of an open neighborhood, $|N(x)|$, around a vertex $x$, denoted by $\zeta_x$, is called its degree. A vertex $x$ is called a leaf if $\zeta_x = 1$. The longest path between any two leaves in a tree is defined as its diameter. A diameter path in a tree $\mathbb{T}$ is denoted by $P_{d+1} = {x_1, x_2, \ldots, x_{d+1}}$, where the path between vertices $x_1, x_2, \ldots, x_{d+1}$ reaches this maximum length.
	
	For a given vertex $y \in \mathbb{X}$, the graph $\mathbb{G} - {y}$ is obtained by removing $y$, which results in a new vertex set $\mathbb{X} - {y}$ and an edge set $\mathbb{Y} - {yx \mid x \in N(y)}$. Similarly, for an edge $e \in \mathbb{Y}$, the graph $\mathbb{G} - {e}$ retains the original vertex set $\mathbb{X}$ but removes the edge $e$, resulting in the edge set $\mathbb{Y} - {e}$.
	
	For $l$ vertices $x_1, \ldots, x_l$ or edges $e_1, \ldots, e_l$, we define the graph $\mathbb{G} - {x_1, \ldots, x_l}$ as $\left(\mathbb{G} - {x_1, \ldots, x_{l-1}}\right) - {x_l}$, and similarly, $\mathbb{G} - {e_1, \ldots, e_l}$ as $\left(\mathbb{G} - {e_1, \ldots, e_{l-1}}\right) - {e_l}$. The path graph $\mathbb{P}_n$, and the star graph $\mathbb{S}_n$. We direct the reader to \cite{west2001introduction} for definitions of any other notation and terminology not covered here.
	
	The $RDN$ of a graph $\mathbb{G}$ is the minimum weight of a Roman dominating function on $\mathbb{G}$. An $RDN$ is a function $\aleph: \mathbb{X}(\mathbb{G}) \rightarrow {0, 1, 2}$ such that every vertex $y \in \mathbb{X}(\mathbb{G})$ with $\aleph(y) = 0$ is adjacent to at least one vertex $x \in \mathbb{X}(\mathbb{G})$ where $\aleph(x) = 2$. The weight of $\aleph$ is the sum $\aleph(\mathbb{X}) = \sum_{y \in \mathbb{X}(\mathbb{G})} \aleph(y)$ \cite{cockayne2004roman}. Essentially, the RDN ensures that any unguarded vertex (with $\aleph(y) = 0$) is adjacent to a heavily guarded vertex (with $\aleph(x) = 2$), and it is denoted by $\Gamma_R$.


	\section{Preliminary Results}
	We introduce a number of lemmas in this section that will be utilized to demonstrate the primary theorem. Using the Mathematics program, all single and double variables function-related inequalities in these lemmas and the main theorem's proof have been confirmed.
	
	\begin{lemma}\label{lem1}
		Suppose $m(a) =(a-1) \sqrt{\frac{a-1}{a}}-(a-2) \sqrt{\frac{a-2}{a-1}}$ with $a \geq 3$. Then, $m(a)$ is increasing function. 
	\end{lemma}
	\begin{proof}
		Suppose that $f(a)=(a-1) \sqrt{\frac{a-1}{a}}$. The derivative is  $f'(a) = \sqrt{\frac{a-1}{a}} + 
		\frac{a-1}{2a^2}\sqrt{\frac{a}{a-1}}$. The expression for \( f'(a) \) is positive for \( a \geq 2 \). Since \( f'(a) > 0 \) in this range, the function is increasing.	Thus, \( f(a) \) is increasing for \( a \geq 2 \). Therefore, $m(a)=f(a)-f(a-1)\geq 0$. 
	\end{proof}
	
	\begin{lemma}\label{lem2}
		Suppose $q(a) = \sqrt{\frac{a+b-2}{ab}}-\sqrt{\frac{a+b-3}{(a-1)b}}$ with $b$, and $a\geq 3$. Then, $q(a)$ is decreasing function for any $b \geq 2$.
	\end{lemma}
	\begin{proof}
		We have that
		$q'(a)= \frac{-b+2}{2 \sqrt{b} a^{\frac{3}{2}} \sqrt{a+b-2}}-\frac{-b+2}{2 \sqrt{b}(a-1)^{\frac{3}{2}} \sqrt{a+b-3}}\geq 0 \iff \frac{1}{ a^{\frac{3}{2}} \sqrt{a+b-2}}\geq \frac{1}{(a-1)^{\frac{3}{2}} \sqrt{a+b-3}}$.
		If we denote $k(b)=\frac{1}{ a^{\frac{3}{2}} \sqrt{a+b-2}}$, since $k'(b)=\frac{-1}{2a^{\frac{3}{2}}(a+b-2)^{\frac{3}{2}}}<0$  for $b\geq 2$, $k(b)$ is a decreasing function. Thus, it follows that $\frac{1}{ a^{\frac{3}{2}} \sqrt{a+b-2}} < \frac{1}{(a-1)^{\frac{3}{2}} \sqrt{a+b-3}}$. Hence, we conclude that $q(a)$ is a decreasing function for $a\geq 3$ and $b\geq 2$.
	\end{proof}
	\begin{lemma}\label{lem3}
		Suppose $\Xi (a,b) =(a-1) \sqrt{\frac{a-1}{a}}+\sqrt{\frac{a+b-2}{ab}}-(a-2) \sqrt{\frac{a-2}{a-1}}-\sqrt{\frac{a+b-3}{(a-1)b}}>\frac{\sqrt{5}}{2\sqrt{2}}$ with $a \geq 3$, and $b\geq 2$.
	\end{lemma}
	\begin{proof}
		We begin by decomposing the function $\Xi (a,b)$ into two parts, $f_{1}(a) =(a-1) \sqrt{\frac{a-1}{a}}-(a-2) \sqrt{\frac{a-2}{a-1}}$ and $f_{2}(a,b) = \sqrt{\frac{a+b-2}{ab}}-\sqrt{\frac{a+b-3}{(a-1)b}}$. By Lemma \ref{lem1} $f_{1}(a)$ is increasing function and it satisfies $f_{1}(a)\geq 0.9258$ ($f_{1}(3)= 0.9258$). By Lemma \ref{lem2} $f_2(a,b)$ is deceasing function and it satisfies $f_{2}(a,b)\geq -0.1296$. Since $\Xi (a,b)=f_1(a,b)+f_2(a,b)$ we conclude $\Xi (a,b) \geq 0.7962 > \frac{\sqrt{5}}{2\sqrt{2}}$. Thus, it follows that $\Xi (a,b)> \frac{\sqrt{5}}{2\sqrt{2}}$. This completes the proof. 
	\end{proof}
	
	\begin{lemma}\label{lem5}
		Suppose that $p(a)=\sqrt{a-b}\sqrt{a-b-1}-\sqrt{a-b+1}\sqrt{a-b}$ with $a\geq 3$ and $b\leq \lceil\frac{2a}{3}\rceil$ is a increasing and negative function, and $-\sqrt{2}\leq p(a)<-1$.
	\end{lemma}
	\begin{proof}
		Suppose that $k(a)= \sqrt{a-b}\sqrt{a-b-1}$ and $k'(a)=\frac{2a-2b-1}{2 \sqrt{a-b-1} \sqrt{a-b}}$, so $k(a)$ is a increasing function for $a\geq 3$ and $b\leq \lceil\frac{2a}{3}\rceil$. Therefor, give function $p(a)=k(a)-k(a+1)$ is a increasing and negative function. We verified that given function hold the inequalities $-\sqrt{2}\leq p(a)<-1$.
	\end{proof}
	\begin{lemma} \label{lem6}
		Let $m(a)=\sqrt{a-b-1}\sqrt{a-b-2}-\sqrt{a-b+1}\sqrt{a-b}$ with $a\geq 4$ and $b\leq \lceil\frac{2a}{3}\rceil$ is a increasing and negative function, and $-\sqrt{6}\leq m(a)<-2$. 
	\end{lemma}
	\begin{proof}
		Derivative of $\alpha(a)=\sqrt{a-b-1}\sqrt{a-b-2}$ is $\alpha^{'}(a)=\frac{2a-2b-3}{2\sqrt{a-b-2}\sqrt{a-b-1}}>0$. Therefore, $\alpha(a)$ is a increasing function for $a\geq 4$ and $b\leq \lceil\frac{2a}{3}\rceil$. Hence, give function $m(a)=\alpha(a)-\alpha(a+2)$ is a increasing and negative function. We verified that given function hold the inequalities $-\sqrt{6}\leq m(a)<-2$.  
	\end{proof}
	
	\begin{theorem}\textnormal{\label{th1} \cite{cockayne2004roman}}
		\textnormal{For the path graph $P_{n}$}, $RDN$ is 
		$\Gamma_{R}(\mathbb{P}_{n})=\left\lceil\frac{2n}{3}\right\rceil.$
	\end{theorem}

	\section{Main Results}
	This section presents the extremal values of the $ABC$ index of trees in terms of their Roman domination number and order. We define two functions, $\mathbb{f}{\text{min}}(n, \Gamma_R)$ and $\mathbb{f}{\text{max}}(n, \Gamma_R)$, which represent the lower and upper bounds of the $ABC$ index for trees, respectively, based on the order $n$ and the $RDN$ $\Gamma_R$. The proofs of these bounds are provided in Theorems \ref{them2} and \ref{them3}. Additionally, Theorems \ref{them4} and \ref{them5} identify specific graphs that achieve these exact values.
	
	\begin{equation*}
		\begin{split}
			\mathbb{f}_{min}(n,\Gamma_R)=& \frac{1}{\sqrt{2}}(n-1)+\left\lceil\frac{2n}{3}\right\rceil\left(\frac{3}{4}-\frac{1}{\sqrt{2}}\right)+\Gamma_R\left(\frac{1}{\sqrt{2}}-\frac{3}{4}\right).\\
			\mathbb{f}_{max}(n,\Gamma_R)=& \sqrt{n-\Gamma_R+1}\sqrt{n-\Gamma_R}-\left(\Gamma_R -2\right)\left(\frac{1}{2}-\frac{3}{\sqrt{5}}\right).
		\end{split}
	\end{equation*}
	\noindent The following lemma assists in establishing the minimum value of the $ABC$ index in terms of the order and its $RDN$.

	\begin{lemma}\label{lem4}
		Suppose that $\mathbb{T}$ is a tree graph and $\Gamma_R$ is a $RDN$, a vertex $x \in V(\mathbb{T})$ such that $\zeta(x)= m \geq 3$, $N(x)=\{y_{1},y_{2},\ldots,y_{m}\}$, $\zeta(y_{m})=j\geq 2$, $\zeta(y_{a})=1$ for every $a \in \{1,2,3, \ldots, m-1\}$. If we take $T'=T-{y_{1}}$, we have: 
	\end{lemma}
	\begin{proof}
		Since $\mathbb{T'}=\mathbb{T}-{y_{1}}$, we have $ABC(\mathbb{T})= ABC(\mathbb{T'})+ (m-1) \sqrt{\frac{m-1}{m}}+\sqrt{\frac{m+j-2}{mj}}-(m-2) \sqrt{\frac{m-2}{m-1}}-\sqrt{\frac{m+j-3}{(m-1)j}}
		\geq \frac{1}{\sqrt{2}}(n-1)-\frac{1}{\sqrt{2}}+\left\lceil\frac{2(n-1)}{3}\right\rceil\left(\frac{3}{4}-\frac{1}{\sqrt{2}}\right)+\Gamma_R\left(\frac{1}{\sqrt{2}}-\frac{3}{4}\right)+(m-1) \sqrt{\frac{m-1}{m}}+\sqrt{\frac{m+j-2}{mj}}
		-(m-2) \sqrt{\frac{m-2}{m-1}}-\sqrt{\frac{m+j-3}{(m-1)j}}
		\geq  \mathbb{f}_{min}(n,\Gamma_R)+(m-1) \sqrt{\frac{m-1}{m}}+\sqrt{\frac{m+j-2}{mj}}-(m-2) \sqrt{\frac{m-2}{m-1}}-\sqrt{\frac{m+j-3}{(m-1)j}}-\frac{1}{\sqrt{2}}-\left(\frac{3}{4}-\frac{1}{\sqrt{2}}\right).$
		Suppose that $\alpha(m,j)=(m-1) \sqrt{\frac{m-1}{m}}+\sqrt{\frac{m+j-2}{mj}}-(m-2) \sqrt{\frac{m-2}{m-1}}-\sqrt{\frac{m+j-3}{(m-1)j}}$. Using Lemma \ref{lem3}, we say $\alpha(m,j)> \frac{\sqrt{5}}{2\sqrt{2}}$. So,
		$ABC(\mathbb{T})\geq \mathbb{f}_{max}(n,\Gamma_R)+\frac{\sqrt{5}}{2\sqrt{2}}-\frac{1}{\sqrt{2}}-\left(\frac{3}{4}-\frac{1}{\sqrt{2}}\right)>\mathbb{f}_{max}(n,\Gamma_R).$
	\end{proof}
	
	\begin{theorem} \label{them2}
		Suppose that $\mathbb{T}$ be a tree graph and let $\Gamma_R$ denote its $RDN$. Then the $ABC$ index of $\mathbb{T}$ satisfies the inequality $ABC(\mathbb{T}) \geq \mathbb{f}_{min}(n, \Gamma_R)$.
	\end{theorem}
	\begin{proof}
		Let us demonstrate the outcome using induction regarding the vertex count. $T$ is either a star $S_{4}$ or a path $P_{4}$. As we have already observed, $ABC(P_{4})=\mathbb{f}_{min}(4, 3)$, and $ABC(S_{4})=2.44>\mathbb{f}_{min}(4, 3)=2.1213$. We examine a $\mathbb{T}$ with order $n$ and $\Gamma_R$. We consider that the inequality holds for every $\mathbb{T}$ with $n-1$ vertices. Now we prove that for when $\mathbb{T}$ has $n$ vertices. We discuss some cases.
		
		\noindent \textbf{Case 1:} We consider that $x-y-z$ is a path in $\mathbb{T}$. If $x$ is a leaf, $\zeta(y)= m\geq 3$ and $\zeta(z) = j\geq2$, then we apply Lemma \ref{lem4} and get the result.
		
		\noindent \textbf{Case 2:} Let $P_{d+1}=\{x_{1}, x_{2}, \ldots, x_{d+1}\}$ is a diametral path of the tree graph. Now, Let degree of $x_{2}$ is 2. \\
		\noindent \textbf{Case 2.1:} Suppose that $d(x_{3})=m\geq 4$, $N(x_{3})=\{x_{2},x_{4},w_{1},\ldots,w_{m-2}\}$, $d(w_{\alpha})=b_{\alpha}\leq 2$, $\alpha=2,\ldots,m-2$, $d(w_{1})=1$ and $d(x_{4})=k$ where $3\leq k \leq m$. If $\mathbb{T}=\mathbb{T}_{1}-\{x_{1},x_{2},w_{1}\}$, then $\Gamma_R(\mathbb{T})= \Gamma_R(\mathbb{T}_{1})+1$, we get:
		$ABC(\mathbb{T})=  ABC(\mathbb{T}_{1})+\sum_{\alpha=1}^{m-3}\left(\frac{1}{\sqrt{b_{\alpha}}}\right)\left(\sqrt{\frac{m+b_{\alpha}-2}{m}}-\sqrt{\frac{m+b_{\alpha}-4}{m-2}}\right)+\frac{1}{\sqrt{k}}\left(\sqrt{\frac{m+k-2}{m}}-\sqrt{\frac{m+k-4}{m-2}}\right)+\sqrt{\frac{m-1}{m}}+\sqrt{2}\geq \frac{1}{\sqrt{2}}(n-1)-\frac{3}{\sqrt{2}}+\left\lceil\frac{2(n-3)}{3}\right\rceil\left(\frac{3}{4}-\frac{1}{\sqrt{2}}\right)+\Gamma_R\left(\frac{1}{\sqrt{2}}-\frac{3}{4}\right)-\left(\frac{1}{\sqrt{2}}-\frac{3}{4}\right)+\sum_{\alpha=1}^{m-3}\left(\frac{1}{\sqrt{b_{\alpha}}}\right)\left(\sqrt{\frac{m+b_{\alpha}-2}{m}}-\sqrt{\frac{m+b_{\alpha}-3}{m-1}}\right)+\frac{1}{\sqrt{k}}\left(\sqrt{\frac{m+k-2}{m}}-\sqrt{\frac{m+k-3}{m-1}}\right)+\sqrt{\frac{m-1}{m}}+\sqrt{2}\geq  \mathbb{f}_{min}(n, \Gamma_R)+\sum_{\alpha=1}^{m-3}\left(\frac{1}{\sqrt{b_{\alpha}}}\right)\left(\sqrt{\frac{m+b_{\alpha}-2}{m}}-\sqrt{\frac{m+b_{\alpha}-4}{m-2}}\right)+\frac{1}{\sqrt{k}}(\sqrt{\frac{m+k-2}{m}}-\sqrt{\frac{m+k-4}{m-2}})+\sqrt{\frac{m-1}{m}}-\left(\frac{3}{4}-\frac{1}{\sqrt{2}}\right)-\frac{3}{\sqrt{2}}+\sqrt{2}\geq  \mathbb{f}_{min}(n, \Gamma_R)+\left(\frac{m-3}{\sqrt{3}}\right)\left(\sqrt{\frac{m+1}{m}}-\sqrt{\frac{m-1}{m-2}}\right)+\frac{1}{\sqrt{k}}\left(\sqrt{\frac{m+k-2}{m}}-\sqrt{\frac{m+k-4}{m-2}}\right)+\sqrt{\frac{m-1}{m}}-\frac{3}{4}$.
		Suppose that $\alpha(m,k) = \left(\frac{m-3}{\sqrt{3}}\right)\left(\sqrt{\frac{m+1}{m}}-\sqrt{\frac{m-1}{m-2}}\right)+\frac{1}{\sqrt{k}}\left(\sqrt{\frac{m+k-2}{m}}-\sqrt{\frac{m+k-4}{m-2}}\right)+\sqrt{\frac{m-1}{m}}-\frac{3}{4}$. $\alpha(m,k)>0$ for $m \geq 4$ and $3\leq k \leq m$. So, $ABC(\mathbb{T})\geq \mathbb{f}_{min}(n, \Gamma_R) + \alpha(m,k)> \mathbb{f}_{min}(n, \Gamma_R)$.\\
		
		\noindent \textbf{Case 2.2:} We assume $\zeta(x_{3})=3$, $N(x_{3})=\{x_{2},x_{4},y_{1}\}$, $y_{1}$ is the leaf of the $\mathbb{T}$ and $\zeta(x_{4})=k$. \\
		\noindent \textbf{Case 2.2.1:} We suppose that $1\leq k \leq 4$. If we take $\mathbb{T}_{2}=\mathbb{T}-\{x_{1},x_{2}\}$, then $\Gamma_R(\mathbb{T})= \Gamma_R(\mathbb{T}_{2})+1$, we get:\\
		\noindent	$ABC(\mathbb{T})=  ABC(\mathbb{T}_{1}) + \sqrt{\frac{k+1}{3k}}+\sqrt{\frac{2}{3}} 
		\geq \frac{1}{\sqrt{2}}(n-1)-\frac{2}{\sqrt{2}}+\left\lceil\frac{2(n-2)}{3}\right\rceil\left(\frac{3}{4}-\frac{1}{\sqrt{2}}\right)+(\Gamma_R-1)\left(\frac{1}{\sqrt{2}}-\frac{3}{4}\right) + \sqrt{\frac{k+1}{3k}}+\sqrt{\frac{2}{3}}
		\geq  \mathbb{f}_{min}(n, \Gamma_R) +\sqrt{\frac{k+1}{3k}}+\sqrt{\frac{2}{3}}-\sqrt{2}-\left(\frac{3}{4}-\frac{1}{\sqrt{2}}\right)$.
		Suppose that  $\beta(k)=\sqrt{\frac{k+1}{3k}}+\sqrt{\frac{2}{3}}-\sqrt{2}-\left(\frac{3}{4}-\frac{1}{\sqrt{2}}\right)$. $\beta(k)>0$ for $1\leq k \leq 4$. Therefore, $ABC(\mathbb{T})\geq \mathbb{f}_{min}(n, \Gamma_R)+\beta(k)>\mathbb{f}_{min}(n, \Gamma_R)$.\\
		\noindent \textbf{Case 2.2.2:} We suppose that $k \geq 5$,  $\zeta(x_{5})= u $, $N(x_{4})=\{x_{3}, x_{5}, y_{1},\ldots,y_{k-2}\}$, $\zeta(y_{1})=1$ and $\zeta(y_{a})=b_{a}\leq 5$ where $a=\{2,3,\ldots,k-2\}$. If we take $\mathbb{T}_{4}=\mathbb{T}-\{x_{1},x_{2},x_{3},x_{4},y_{1}\}$, then $\Gamma_R(\mathbb{T})= \Gamma_R(\mathbb{T}_{4})+3$, we get:\\
		$ABC(\mathbb{T})=  ABC(\mathbb{T}_{4})+ \sum_{a=1}^{u-2} \sqrt{\frac{k+b_{a}-2}{kb_{a}}}-\sum_{a=1}^{u-3} \sqrt{\frac{k+b_{a}-3}{\left(k-1\right)b_{a}}}+\sqrt{\frac{k-1}{k}}+\sqrt{\frac{k+1}{3k}}+\sqrt{\frac{k+u-2}{ku}}-\sqrt{\frac{k+u-3}{(k-1)u}}+\frac{2}{\sqrt{2}}+\sqrt{\frac{2}{3}}
		\geq \frac{1}{\sqrt{2}}(n-1)-\frac{5}{\sqrt{2}}+\left\lceil\frac{2(n-5)}{3}\right\rceil\left(\frac{3}{4}-\frac{1}{\sqrt{2}}\right)+(\Gamma_R-3)\left(\frac{1}{\sqrt{2}}-\frac{3}{4}\right)+\sum_{a=1}^{u-2} \sqrt{\frac{k+b_{a}-2}{kb_{a}}}-\sum_{a=1}^{u-3} \sqrt{\frac{k+b_{a}-3}{\left(k-1\right)b_{a}}}
		+\sqrt{\frac{k+u-2}{ku}}-\sqrt{\frac{k+u-3}{(k-1)u}}+\sqrt{\frac{k-1}{k}}+\sqrt{\frac{k+1}{3k}}+\frac{2}{\sqrt{2}}+\sqrt{\frac{2}{3}}
		\geq  \mathbb{f}_{min}(n, \Gamma_R)+ \sum_{a=1}^{u-2} \sqrt{\frac{k+b_{a}-2}{kb_{a}}}-\sum_{a=1}^{u-3} \sqrt{\frac{k+b_{a}-3}{\left(k-1\right)b_{a}}}+\sqrt{\frac{k-1}{k}}+\sqrt{\frac{k+1}{3k}}+\sqrt{\frac{k+u-2}{ku}}-\sqrt{\frac{k+u-3}{(k-1)u}}
		+\frac{2}{\sqrt{2}}+\sqrt{\frac{2}{3}}-\frac{5}{\sqrt{2}}
		\geq  \mathbb{f}_{min}(n, \Gamma_R)+ (k-2) \sqrt{\frac{k+3}{5k}}-(k-3) \sqrt{\frac{k+2}{5\left(k-1\right)}}+\sqrt{\frac{k-1}{k}}+\sqrt{\frac{k+1}{3k}}+\sqrt{\frac{k+u-2}{ku}}-\sqrt{\frac{k+u-3}{(k-1)u}}
		+\frac{2}{\sqrt{2}}+\sqrt{\frac{2}{3}}-\frac{5}{\sqrt{2}}$\\
		Suppose that $\alpha(k)=(k-2) \sqrt{\frac{k+3}{5k}}-(k-3) \sqrt{\frac{k+2}{5\left(k-1\right)}}+\sqrt{\frac{k-1}{k}}+\sqrt{\frac{k+1}{3k}}$. It has been verified that  $\alpha(k)>1.57$. Now, Let  $\beta(k,u)=\sqrt{\frac{k+u-2}{ku}}-\sqrt{\frac{k+u-3}{(k-1)u}}$. By using Lemma \ref{lem2}, we have $\beta(k,u)>-0.013$. Therefore, $ABC(\mathbb{T})\geq \mathbb{f}_{min}(n, \Gamma_R)+\alpha(k)+\beta(k,u)+\frac{2}{\sqrt{2}}+\sqrt{\frac{2}{3}}-\frac{5}{\sqrt{2}}>\mathbb{f}_{min}(n, \Gamma_R)$.\\
		
		\noindent \textbf{Case 2.3:} Assume that $\zeta(x_3) = 2$, $\zeta(x_4) = l\geq 2$, and $N(x_4) = \{x_3, x_5, y_1, \ldots, y_{l-2}\}$, where $\zeta(y_i) = a_i \leq 5$ for $i \in \{1, 2, \ldots, l-2\}$ and $\zeta(x_5) = t\leq l$. If we take $\mathbb{T}_{5}=\mathbb{T}-\{x_{1},x_{2},x_{3},y_{1}\}$, then $\Gamma_R(\mathbb{T})= \Gamma_R(\mathbb{T}_{4})+2$, we get:\\
		$ABC(\mathbb{T})=  ABC(\mathbb{T}_{5})+ \sum_{i=1}^{l-3} \frac{1}{\sqrt{a_{i}}}\left(\sqrt{\frac{l+a_{i}-2}{l}}-\sqrt{\frac{l+a_{i}-3}{l-1}}\right)+\sqrt{\frac{l-1}{l}}+\frac{1}{\sqrt{k}}\left(\sqrt{\frac{l+k-2}{l}}-\sqrt{\frac{l+k-3}{l-1}}\right)+\frac{3}{\sqrt{2}}
		\geq  \frac{1}{\sqrt{2}}(n-1)-\frac{4}{\sqrt{2}}+\left\lceil\frac{2(n-4)}{3}\right\rceil\left(\frac{3}{4}-\frac{1}{\sqrt{2}}\right)+(\Gamma_R-2)\left(\frac{1}{\sqrt{2}}-\frac{3}{4}\right)+\sum_{i=1}^{l-3} \frac{1}{\sqrt{a_{i}}}\left(\sqrt{\frac{l+a_{i}-2}{l}}-\sqrt{\frac{l+a_{i}-3}{l-1}}\right)
		+\sqrt{\frac{l-1}{l}}+\frac{1}{\sqrt{k}}\left(\sqrt{\frac{l+k-2}{l}}-\sqrt{\frac{l+k-3}{l-1}}\right)+\frac{3}{\sqrt{2}}
		\geq  \mathbb{f}_{min}(n, \Gamma_R)+\sum_{i=1}^{l-3} \frac{1}{\sqrt{a_{i}}}\left(\sqrt{\frac{l+a_{i}-2}{l}}-\sqrt{\frac{l+a_{i}-3}{l-1}}\right)+\sqrt{\frac{l-1}{l}}+\frac{1}{\sqrt{k}}\left(\sqrt{\frac{l+k-2}{l}}-\sqrt{\frac{l+k-3}{l-1}}\right)-\frac{1}{\sqrt{2}}
		-\left(\frac{3}{4}-\frac{1}{\sqrt{2}}\right)
		>  \mathbb{f}_{min}(n, \Gamma_R)+\frac{l-3}{\sqrt{5}}\left(\sqrt{\frac{l+3}{l}}-\sqrt{\frac{l+2}{l-1}}\right)+\sqrt{1-\frac{1}{l}}+\frac{1}{\sqrt{k}}\left(\sqrt{\frac{l+k-2}{l}}-\sqrt{\frac{l+k-3}{l-1}}\right)-\frac{3}{4}$.\\
		Let $\alpha(l)=\frac{l-3}{\sqrt{5}}\left(\sqrt{\frac{l+3}{l}}-\sqrt{\frac{l+2}{l-1}}\right)+\sqrt{1-\frac{1}{l}}$. Since $\alpha(l)$ is an increasing function for every $l\geq 3$, we define 
		$\beta(l,k)= \frac{1}{\sqrt{k}} \left(\sqrt{\frac{l+k-2}{l}}-\sqrt{\frac{l+k-3}{l-1}}\right)$. By Lemma \ref{lem2} $\beta(l,k)$ is a decreasing function. Therefore, the inequality holds:  $ABC(\mathbb{T})\geq \mathbb{f}_{min}(n, \Gamma_R)+\alpha(l)+\beta(l,k)-\frac{3}{4}>\mathbb{f}_{min}(n, \Gamma_R)$.
	\end{proof}

	\begin{theorem} \label{them3}
		If $\mathbb{T}$ be a tree graph and let $\Gamma_R$ denote its $RDN$. Then the ABC index of $\mathbb{T}$ satisfies the inequality $ABC(\mathbb{T}) \leq \mathbb{f}_{max}(n, \Gamma_R)$.
	\end{theorem}		
	
	\begin{proof}
		Let us demonstrate the outcome using induction regarding the vertex count. $\mathbb{T}$ is either a star $S_{4}$ or a path $P_{4}$. As we have already observed, $ABC(P_{4})=2.121<\mathbb{f}_{max}(4, 3)$, and $ABC(S_{4})=  \mathbb{f}_{max}(4, 3)=2.44$. We examine a $\mathbb{T}$ with order $n$ and $RDN$ $\Gamma_R$. We consider that the inequality holds for every $\mathbb{T}$ with $n-1$ vertices. Now we prove that for when $\mathbb{T}$ has $n$ vertices. We discuss some cases.
		
		\noindent \textbf{Case 1:} Consider a vertex $x$ where $\zeta(x)=m\geq 3$, $N(x)=\{y_{1},y_{2},\ldots,y_{m}\}$, $\zeta(y_{m})=j\geq 2$, $\zeta(y_{a})=1$ for every $a \in \{1,2,3, \ldots, m-1\}$. If we take $T'=T-{y_{1}}$, then $\Gamma_R(\mathbb{T})=\Gamma_R(\mathbb{T'})$. We have:\\
		$ABC(\mathbb{T})= ABC(\mathbb{T'})+ \frac{1}{\sqrt{m}} \left((m-1)^{\frac{3}{2}}+\sqrt{\frac{m+j-2}{j}}\right)-\frac{1}{m-1} \left((m-2)^{\frac{3}{2}}+\sqrt{\frac{m+j-3}{j}}\right)
		\leq  \sqrt{n-\Gamma_R}\sqrt{n-\Gamma_R-1}-\left(\Gamma_R -2\right)\left(\frac{1}{2}-\frac{3}{\sqrt{5}}\right)+ \frac{1}{\sqrt{m}} \left((m-1)^{\frac{3}{2}}+\sqrt{\frac{m+j-2}{j}}\right)-\frac{1}{\sqrt{m-1}} \left((m-2)^{\frac{3}{2}}+\sqrt{\frac{m+j-3}{j}}\right)\leq  \mathbb{f}_{max}(n,\Gamma_R)+\sqrt{n-\Gamma_R}\\ \sqrt{n-\Gamma_R-1}-\sqrt{n-\Gamma_R+1}\sqrt{n-\Gamma_R}+ \frac{1}{\sqrt{m}} \left((m-1)^{\frac{3}{2}}+\sqrt{\frac{m+j-2}{j}}\right)-\frac{1}{\sqrt{m-1}} \left((m-2)^{\frac{3}{2}}+\sqrt{\frac{m+j-3}{j}}\right)$.
		We assume that $\alpha(n)=\sqrt{n-\Gamma_R}\sqrt{n-\Gamma_R-1}-\sqrt{n-\Gamma_R+1}\sqrt{n-\Gamma_R}$ and $\beta(m,j)=\frac{1}{\sqrt{m}} \left((m-1)^{\frac{3}{2}}+\sqrt{\frac{m+j-2}{j}}\right)-\frac{1}{\sqrt{m-1}} \left((m-2)^{\frac{3}{2}}+\sqrt{\frac{m+j-3}{j}}\right)$. By using Lemma \ref{lem5}, we get $-\sqrt{2}\leq\alpha(n)<-1$ for $n\geq 3$ and $2\leq\Gamma_R\leq \lceil\frac{2n}{3}\rceil$ and by Lemma \ref{lem3}, we get $\frac{\sqrt{5}}{2\sqrt{2}}\leq \beta(m,j)<1$ for any $m\geq 3$ and $j\geq 2$. Therefore, $\eta(n,m,j)=\alpha(n)+\beta(m,j)$ is negative function. So, $	ABC(\mathbb{T})\leq \mathbb{f}_{max}(n,\Gamma_R)+\eta(n,m,j)<\mathbb{f}_{max}(n,\Gamma_R)$.\\
		
		\noindent \textbf{Case 2:} Let $P_{d+1}=\{x_{1},x_{2},\ldots,x_{d+1}\}$ is a diametral path of the tree. Let $\zeta(x_{2})=2$, $\zeta(x_{3})=m$, $N(x_{3})=\{x_{2},x_{4},y_{1},\ldots,y_{m-2}\}$, $\zeta(y_{i})=b_{i}$, where $i\in \{1,2,\ldots,m-2\}$, and $\zeta(x_{4})=k$. If we take $\mathbb{T}_{2}=\mathbb{T}-\{x_{1},x_{2}\}$, then $\Gamma_R(\mathbb{T})=\Gamma_R(\mathbb{T}_{2})+1$, we get:
		$ABC(\mathbb{T})= ABC(\mathbb{T}_{2})+\frac{1}{\sqrt{k}}\left(\sqrt{\frac{m+k-2}{m}}-\sqrt{\frac{m+k-3}{m-1}}\right)+\sum_{i=1}^{m-1}\frac{1}{\sqrt{b_{i}}}\left(\sqrt{\frac{m+b_{i}-2}{m}}-\sqrt{\frac{m+b_{i}-3}{m-1}}\right)+\sqrt{2}
		\leq  \sqrt{n-\Gamma_R-1}\sqrt{n-\Gamma_R-2}-\left(\Gamma_R -3\right)\left(\frac{1}{2}-\frac{3}{\sqrt{5}}\right)+\frac{1}{\sqrt{k}}(\sqrt{\frac{m+k-2}{m}}-\sqrt{\frac{m+k-3}{m-1}})+\sum_{i=1}^{m-2}\frac{1}{\sqrt{b_{i}}}\left(\sqrt{\frac{m+b_{i}-2}{m}}-\sqrt{\frac{m+b_{i}-3}{m-1}}\right)+\sqrt{2}\leq  \mathbb{f}_{max}(n,\Gamma_R)+\sqrt{n-\Gamma_R-1}\sqrt{n-\Gamma_R-2}-\sqrt{n-\Gamma_R+1} \sqrt{n-\Gamma_R}+\frac{1}{\sqrt{k}}\left(\sqrt{\frac{m+k-2}{m}}-\sqrt{\frac{m+k-3}{m-1}}\right)+\sum_{i=1}^{m-2}\\ \frac{1}{\sqrt{b_{i}}}\left(\sqrt{\frac{m+b_{i}-2}{m}}-\sqrt{\frac{m+b_{i}-3}{m-1}}\right)+\sqrt{2}+\left(\frac{1}{2}-\frac{3}{\sqrt{5}}\right)
		\leq  \mathbb{f}_{max}(n,\Gamma_R)+\sqrt{n-\Gamma_R-1}\sqrt{n-\Gamma_R-2}-\sqrt{n-\Gamma_R+1}\sqrt{n-\Gamma_R}+\frac{1}{\sqrt{k}}\left(\sqrt{\frac{m+k-2}{m}}-\sqrt{\frac{m+k-3}{m-1}}\right)
		+\frac{m-2}{\sqrt{5}}\left(\sqrt{\frac{m+3}{m}}-\sqrt{\frac{m+2}{m-1}}\right)+\sqrt{2}+\left(\frac{1}{2}-\frac{3}{\sqrt{5}}\right)$.\\
		We assume that $\alpha(n)=\sqrt{n-\Gamma_R-1}\sqrt{n-\Gamma_R-2}-\sqrt{n-\Gamma_R+1}\sqrt{n-\Gamma_R}$,  $\beta(m,k)=\frac{1}{\sqrt{k}}\left(\sqrt{\frac{m+k-2}{m}}-\sqrt{\frac{m+k-3}{m-1}}\right)$ and $\gamma(m)=\frac{m-2}{\sqrt{5}}\left(\sqrt{\frac{m+3}{m}}-\sqrt{\frac{m+2}{m-1}}\right)$. By using Lemma \ref{lem6}, we get $-\sqrt{6}\leq\alpha(n)<-2$, by Lemma \ref{lem2} $\beta(m,k)$ is decreasing function, and $\gamma(m)$ is negative function for any $m\geq 3$. Hence, $ABC(\mathbb{T}) \leq  \mathbb{f}_{max}(n,\Gamma_R)+\alpha(n)+\beta(m,k)+\gamma(m)+\sqrt{2}+\left(\frac{1}{2}-\frac{3}{\sqrt{5}}\right)<\mathbb{f}_{max}(n,\Gamma_R)$.
		
	\end{proof}
	
	\begin{theorem} \label{them4}
		\textnormal{Suppose that $\mathbb{T}$ be a tree with order $n$ and $RDN$ $ \Gamma_{R}$, then we have $ABC(\mathbb{T})= \mathbb{f}_{min}(n,\Gamma_R)$
			if and only if $\mathbb{T} = \mathbb{P}_{n}$}.
	\end{theorem}
	\begin{proof}
		Using Equation \ref{equ1}, we know that for the path graph $\mathbb{P}_{n}$, the  $ABC$ index is: $ABC(\mathbb{P}_{n})=\frac{1}{\sqrt{2}}(n-1)$. By Theorem \ref{th1}, the $RDN$ of the path graph $\mathbb{P}_{n}$ is, $\Gamma_{R}(\mathbb{P}_{n})=\left\lceil\frac{2n}{3}\right\rceil$. Substituting this into the formula for $\mathbb{f}_{max}(n,\Gamma_R)$, we get $\mathbb{f}_{max}(n,\Gamma_R)=\frac{1}{\sqrt{2}}(n-1)+\left\lceil\frac{2n}{3}\right\rceil\left(\frac{3}{4}-\frac{1}{\sqrt{2}}\right)+\Gamma_R(\mathbb{P}_{n})\left(\frac{1}{\sqrt{2}}-\frac{3}{4}\right)=\frac{1}{\sqrt{2}}(n-1)=ABC(\mathbb{P}_{n})$.
	\end{proof}
	\begin{theorem} \label{them5}
		\textnormal{Suppose that $\mathbb{T}$ a tree with order $n$ and the $RDN$ $ \Gamma_{R}$, then we have $ABC(\mathbb{T})= \mathbb{f}_{max}(n,\Gamma_R)$ if and only if $\mathbb{T}= S_{n}$.}
	\end{theorem}
	\begin{proof} 
		By using Equation \ref{equ1}, we know that for the star graph $\mathbb{S}_{n}$, the  $ABC$ index in the term of order of $\mathbb{S}_{n}$ is given by: $ABC(\mathbb{S}_{n})=\sqrt{n-1}\sqrt{n-2}$ and $RDN$ of $\mathbb{S}_{n}$ is 2 ($\Gamma_{R}(\mathbb{S}_{n})=2$). Hence, we substitute this into the formula for $\mathbb{f}_{max}(n,\Gamma_R(\mathbb{S}_{n}))$, giving: $\mathbb{f}_{max}(n,\Gamma_R(\mathbb{S}_{n}))=\sqrt{n-1}\sqrt{n-2}=ABC(\mathbb{S}_{n})$. This is exactly equal to the $ABC$ index of $\mathbb{S}_{n}$, so we conclude that: $ABC(\mathbb{S}_{n})=\mathbb{f}_{max}(n,\Gamma_R(\mathbb{S}_{n}))$.
	\end{proof}
	
	\section{Conclusion}
	
	In this study, we examined the $ABC$ index of tree graphs, focusing on its dependence on the $RDN$ and the tree's order. We established the lower and upper bounds of the $ABC$ index, as detailed in Theorems \ref{them2} and \ref{them3}, demonstrating that these bounds are determined by the tree’s order and $RDN$. Specifically, Theorem \ref{them4} shows that $\mathbb{P}_{n}$ achieve the minimum $ABC$ index, while Theorem \ref{them5} reveals that $\mathbb{S}_{n}$ trees attain the maximum index, illustrating the significant impact of the $RDN$ on these extremal values. These results enhance the understanding of the relationship between tree parameters and topological indices, providing a basis for future research to explore these interactions in broader graph classes and uncover new connections between the $ABC$ index and other graph invariants.

	\subsection*{Declarations}\noindent\textbf{Author Contribution Statement} All authors contributed equally to the paper.\\
	
	\noindent\textbf{Declaration of competing interest} The authors have no conflict of interest to disclose.\\
	
	\noindent\textbf{Data availability statements} All the data used to find the results is included in the manuscript.\\
	
	\noindent\textbf{Acknowledgment} This research was supported by the Ministry of Higher Education (MOHE) through the Fundamental Research Grant Scheme (FRGS/1/2022/STG06/UMT/03/4).\\
	
	\noindent\textbf{Ethical statement} This article contains no studies with humans or animals.
	\bibliographystyle{plain}
	\bibliography{biblo}	

\begin{thebibliography}{10}

\bibitem{ahmad2023zagreb}
Ayu Ameliatul~Shahilah Ahmad~Jamri, Roslan Hasni, and Sharifah~Kartini
  Said~Husain.
\newblock On the zagreb indices of graphs with given roman domination number.
\newblock {\em Communications in Combinatorics and Optimization},
  8(1):141--152, 2023.
\newblock https://doi.org/10.22049/cco.2021.27439.1263.

\bibitem{ahmad2022minimum}
Ayu Ameliatul~Shahilah Ahmad~Jamri, Fateme Movahedi, Roslan Hasni, Rudrusamy
  Gobithaasan, and Mohammad~Hadi Akhbari.
\newblock Minimum randi{\'c} index of trees with fixed total domination number.
\newblock {\em Mathematics}, 10(20):3729, 2022.
\newblock https://doi.org/10.3390/math10203729.

\bibitem{bermudo2024geometric}
Sergio Bermudo, Roslan Hasni, Fateme Movahedi, and Juan~E N{\'a}poles.
\newblock The geometric--arithmetic index of trees with a given total
  domination number.
\newblock {\em Discrete Applied Mathematics}, 345:99--113, 2024.
\newblock https://doi.org/10.1016/j.dam.2023.11.024.

\bibitem{bermudo2020extremal}
Sergio Bermudo, Juan~E N{\'a}poles, and Juan Rada.
\newblock Extremal trees for the randi{\'c} index with given domination number.
\newblock {\em Applied Mathematics and Computation}, 375:125122, 2020.
\newblock https://doi.org/10.1016/j.amc.2020.125122.

\bibitem{chen2012atom}
Jin-song Chen and Xiao-feng Guo.
\newblock The atom-bond connectivity index of chemical bicyclic graphs.
\newblock {\em Applied Mathematics-A Journal of Chinese Universities},
  27(2):243--252, 2012.
\newblock https://doi.org/10.1007/s11766-012-2756-4.

\bibitem{chen2012some}
Jinsong Chen, Jianping Liu, and Xiaofeng Guo.
\newblock Some upper bounds for the atom--bond connectivity index of graphs.
\newblock {\em Applied Mathematics Letters}, 25(7):1077--1081, 2012.
\newblock https://doi.org/10.1016/j.aml.2012.03.021.

\bibitem{chen2018solution}
Xiaodan Chen and Kinkar~Ch Das.
\newblock Solution to a conjecture on the maximum abc index of graphs with
  given chromatic number.
\newblock {\em Discrete Applied Mathematics}, 251:126--134, 2018.
\newblock https://doi.org/10.1016/j.dam.2018.05.063.

\bibitem{cockayne2004roman}
Ernie~J Cockayne, Paul~A Dreyer~Jr, Sandra~M Hedetniemi, and Stephen~T
  Hedetniemi.
\newblock Roman domination in graphs.
\newblock {\em Discrete mathematics}, 278(1-3):11--22, 2004.
\newblock https://doi.org/10.1016/j.disc.2003.06.004.

\bibitem{das2016comparison}
Kinkar~Ch Das, Mohanad~A Mohammed, Ivan Gutman, and Kamel~A Atan.
\newblock Comparison between atom-bond connectivity indices of graphs.
\newblock {\em MATCH Commun. Math. Comput. Chem}, 76(1):159--170, 2016.
\newblock http://psasir.upm.edu.my/id/eprint/29176.

\bibitem{das2010comparison}
Kinkar~Ch Das and Nenad Trinajsti{\'c}.
\newblock Comparison between first geometric--arithmetic index and atom-bond
  connectivity index.
\newblock {\em Chemical physics letters}, 497(1-3):149--151, 2010.
\newblock https://doi.org/10.1016/j.cplett.2010.07.097.

\bibitem{estrada1998atom}
Ernesto Estrada, Luis Torres, Lissette Rodriguez, and Ivan Gutman.
\newblock An atom-bond connectivity index: modelling the enthalpy of formation
  of alkanes.
\newblock {\em Indian journal of chemistry}, 37(A):849--855, 1998.
\newblock http://nopr.niscpr.res.in/handle/123456789/40308.

\bibitem{furtula2016atom}
Boris Furtula.
\newblock Atom-bond connectivity index versus graovac-ghorbani analog.
\newblock {\em Match Communications in Mathematical and in Computer Chemistry},
  75:233--242, 2016.
\newblock https://scidar.kg.ac.rs/handle/123456789/17398.

\bibitem{gutman1972graph}
Ivan Gutman and Nenad Trinajsti{\'c}.
\newblock Graph theory and molecular orbitals. total $\varphi$-electron energy
  of alternant hydrocarbons.
\newblock {\em Chemical physics letters}, 17(4):535--538, 1972.
\newblock https://doi.org/10.1016/0009-2614(72)85099-1.

\bibitem{hasni2021randic}
Roslan Hasni, Ayu Ameliatul~Shahilah Ahmad~Jamri, Nabeel~Ezzuldin Arif, and
  Fatimah~Noor Harun.
\newblock The randic index of trees with given total domination number.
\newblock {\em Iranian Journal of Mathematical Chemistry}, 12(4):225--237,
  2021.
\newblock https://doi.org/10.22052/ijmc.2021.243170.1591.

\bibitem{karelson1996quantum}
Mati Karelson, Victor~S Lobanov, and Alan~R Katritzky.
\newblock Quantum-chemical descriptors in qsar/qspr studies.
\newblock {\em Chemical reviews}, 96(3):1027--1044, 1996.
\newblock https://doi.org//10.1021/cr950202r.

\bibitem{kier2012molecular}
Lemont Kier.
\newblock {\em Molecular connectivity in chemistry and drug research},
  volume~14.
\newblock Elsevier, 2012.
\newblock https://doi.org/10.1002/jps.2600660852.

\bibitem{kier1986molecular}
Lemont~Burwell Kier and Lowell~H Hall.
\newblock Molecular connectivity in structure-activity analysis.
\newblock {\em Journal of Pharmaceutical Sciences}, 76(3):269--270, 1986.
\newblock https://doi.org/10.1002/jps.2600760325.

\bibitem{li2024greatest}
Fengwei Li, Qingfang Ye, and Huajing Lu.
\newblock The greatest values for atom-bond sum-connectivity index of graphs
  with given parameters.
\newblock {\em Discrete Applied Mathematics}, 344:188--196, 2024.
\newblock https://doi.org/10.1016/j.dam.2023.11.029.

\bibitem{lin2013minimal}
Wenshui Lin, Tianyi Gao, Qi’an Chen, and Xin Lin.
\newblock On the minimal abc index of connected graphs with given degree
  sequence.
\newblock {\em MATCH Commun. Math. Comput. Chem}, 69(3):571--578, 2013.

\bibitem{lin2017minimal}
Wenshui Lin, Peixi Li, Jianfeng Chen, Chi Ma, Yuan Zhang, and Dongzhan Zhang.
\newblock On the minimal abc index of trees with k leaves.
\newblock {\em Discrete Applied Mathematics}, 217:622--627, 2017.
\newblock https://doi.org/10.1016/j.dam.2016.10.007.

\bibitem{noureen2024tricyclic}
Sadia Noureen, Rimsha Batool, Abeer~M Albalahi, Yilun Shang, Tariq Alraqad, and
  Akbar Ali.
\newblock On tricyclic graphs with maximum atom--bond sum--connectivity index.
\newblock {\em Heliyon}, 10(14), 2024.
\newblock https://doi.org/10.1016/j.heliyon.2024.e33841.

\bibitem{todeschini2008handbook}
Roberto Todeschini and Viviana Consonni.
\newblock {\em Handbook of molecular descriptors}.
\newblock John Wiley \& Sons, 2000.
\newblock DOI:10.1002/9783527613106.

\bibitem{vassilev2012minimum}
Tzvetalin~S Vassilev and Laura~J Huntington.
\newblock On the minimum abc index of chemical trees.
\newblock {\em Appl. Math}, 2(1):8--16, 2012.
\newblock https://doi.org/10.1016/j.dam.2024.01.025.

\bibitem{wang2018maximizing}
Shaohui Wang, Chunxiang Wang, Lin Chen, Jia-Bao Liu, and Zehui Shao.
\newblock Maximizing and minimizing multiplicative zagreb indices of graphs
  subject to given number of cut edges.
\newblock {\em Mathematics}, 6(11):227, 2018.
\newblock https://doi.org/10.3390/math6110227.

\bibitem{wang2024maximum}
Zhen Wang and Kai Zhou.
\newblock On the maximum atom-bond sum-connectivity index of unicyclic graphs
  with given diameter.
\newblock {\em AIMS Mathematics}, 9(8):22239--22250, 2024.
\newblock 10.3934/math.20241082.

\bibitem{west2001introduction}
Douglas~B West.
\newblock Introduction to graph theory, 2001.

\bibitem{xu2012relationships}
Xinli Xu.
\newblock Relationships between harmonic index and other topological indices.
\newblock {\em Appl. Math. Sci}, 6(41):2013--2018, 2012.

\bibitem{zhang2016maximum}
Xiu-Mei Zhang, Yu~Yang, Hua Wang, and Xiao-Dong Zhang.
\newblock Maximum atom-bond connectivity index with given graph parameters.
\newblock {\em Discrete Applied Mathematics}, 215:208--217, 2016.
\newblock https://doi.org/10.1016/j.dam.2016.06.021.

\bibitem{zhang2024extremal}
Yuan Zhang, Haiying Wang, Guifu Su, and Kinkar~Chandra Das.
\newblock Extremal problems on the atom-bond sum-connectivity indices of trees
  with given matching number or domination number.
\newblock {\em Discrete Applied Mathematics}, 345:190--206, 2024.
\newblock https://doi.org/10.1016/j.dam.2023.11.046.

\bibitem{zheng2020bounds}
Ruiling Zheng, Jianping Liu, Jinsong Chen, and Bolian Liu.
\newblock Bounds on the general atom-bond connectivity indices.
\newblock {\em MATCH Commun. Math. Comput. Chem}, 83(1):143--166, 2020.
\newblock https://doi.org/10.1016/j.jmaa.2019.06.069.

\end{thebibliography}
\end{document}